\documentclass[11pt]{article}
\usepackage{amsmath,amssymb,amsthm}
\usepackage{mathrsfs}
\usepackage{geometry}
\usepackage{graphicx}
\usepackage{enumerate}
\usepackage{enumitem}
\newtheorem{thm}{Theorem}[section]
\newtheorem{lm}[thm]{Lemma}
\newtheorem{Def}{Definition}[section]
\newtheorem{conj}{Conjecture}[section]
\newtheorem{prob}{Problem}[section]
\newtheorem{pro}[thm]{Proposition}

\newtheorem{cla}{Claim}

\newtheorem{cor}[thm]{Corollary}

\usepackage{latexsym,amsmath,amssymb,amsfonts,epsfig,graphicx,cite,psfrag}
\usepackage{eepic,color,colordvi,amscd}
\usepackage{ebezier}
\usepackage{verbatim}

\usepackage{subfigure}

\newcommand{\pf}{\noindent {\bf Proof.} }

\begin{document}

\title{Density conditions for $k$ vertex-disjoint triangles in tripartite graphs}
\date{}

\author{Mingyang Guo, Klas Markstr\"om}
\maketitle

\begin{abstract}
Let $n,k$ be positive integers such that $n\geq k$ and $G$ be a tripartite graph with parts $A,B,C$ such that $|A|=|B|=|C|=n$. Denote the edge densities of $G[A,B]$, $G[A,C]$ and $G[B,C]$ by $\alpha$, $\beta$ and $\gamma$, respectively. In this paper, we study  edge density conditions for the existence of $k$ vertex-disjoint triangles in a tripartite graph. For $n\geq 5k+2$ we give an optimal condition in terms of densities $\alpha,\beta,\gamma$ for the existence of $k$ vertex-disjoint triangles in $G$. We also give an optimal condition in terms of densities $\alpha,\beta,\gamma$ for the existence of a triangle-factor in $G$.
\end{abstract}

\section{Introduction}
Extremal problems for triangles in graphs are among the most well-studied in graph theory; starting with  Mantel's theorem, showing that any graph with more than $\frac{n^2}{4}$ edges has a triangle.  After Mantel's result  more detailed results for graphs above the existence threshold were sought.  The first such was a result by Rademacher,  who found the minimum number of triangles in a graph with $\frac{n^2}{4}+1$  edges.  Motivated by Rademacher's theorem Erd\H{os} \cite{MR81469} asked for the minimum number of triangles as a function of the number of edges, a problem which became known as the Erd\H{os}-Rademacher problem.  A long  line of papers culminated with Razborov \cite{R2} who found the asymptotic form for the minimum possible number of triangles in a graph with a given number of edges.  Nikiforov generalised this to $K_4$s instead of triangles, and Reiher  \cite{MR3549620} to cliques of any fixed size.   Liu, Pikhurko and Staden \cite{MR4089395} sharpened Razborov's result to give the exact value for the minimum number of triangles and the structure of the extremal graphs. Later \cite{MR4196784} found the structure of the extremal graphs for all cliques.   Generalising in another direction \cite{FMZ21}, investigated the number of edges required to guarantee that some vertex lies in at least $t$ triangles,  and determined the exact number of edges for  tripartite graphs.  
 
 Instead, looking for disjoint triangles,  Erd\H os \cite{E62} studied the edge density condition for the existence of $k$ vertex-disjoint triangles. Let $EX(n,F)$ denote the collection of all $n$-vertex $F$-free graphs with the maximum number of edges. Denote the balanced $r$-partite graph with $n$ vertices by $T(n,r)$. Given two graphs $G$ and $H$ with disjoint vertex set, we use $G\vee H$ to denote the graph with vertex set $V(G)\cup V(H)$ and edge set $\{xy:x\in V(G),y\in V(H)\}$. Let $kF$ be the graph  which consist of $k$ vertex-disjoint copies of $F$.
\begin{thm}[Erd\H os \cite{E62}]
	Suppose that $n,k$ are positive integers satisfying $n\geq  400 k^2$. Then 
	\begin{equation}\label{density-graph}
		EX(n,(k+1)K_3)=\{K_t\vee T(n-k,2)\}.
	\end{equation}
\end{thm}
Later Moon \cite{M68} improved the result by showing that (\ref{density-graph}) holds for $n\geq 9k/2+4$. In 2015, Allen, B\"ottcher, Hladk\'y and Piguet \cite{ABHP15}  extended the result to  $n\geq 3k$ for sufficiently  large $n$ and their result shows that there are four different extremal constructions, for four different ranges of $k$. Recently, Hou, Hu, Liu and Zhang \cite{HHLZ25} determined asymptotically  the edge density condition for the existence of $k$ vertex-disjoint $K_4$. Simonovits \cite{S68} and Moon \cite{M68} showed that if $n$ is sufficiently large and $t$ is fixed, then $K_t\vee T(n-k,r)$ is the unique extremal graph for $ kK_{r+1}$.

Motivated by a question of Erd\H{o}s on triangle-free subgraphs  Bondy, Shen,Thomass\'e and Thomassen \cite{BSTT06}  investigated a tripartite variation of Mantel's theorem. Let $G=(V,E)$ be a tripartite with $3$-partition $V=V_1\cup V_2\cup V_3$. A \textit{weighted tripartite graph} $(G,w)$ is a tripartite graph $G=(V,E)$ together with a \textit{weighting} $w:V \rightarrow [0,1]$ satisfying $$\sum_{a\in V_1}w(a)=\sum_{b\in V_2}w(b)=\sum_{c\in V_3}w(c)=1.$$ The \textit{weight} of an edge $xy\in E(G)$ is $w(xy)=w(x)w(y)$. Let $G[V_i,V_{i+1}]$ denote the bipartite graph induced by $V_i\cup V_{i+1}$ for $i\in [3]$. The \textit{edge density} between two parts of $(G,w)$ is $d(G[V_i,V_j])=\sum_{ab\in E(G[V_i,V_{i+1}])}w(ab)$ for $i\in [3]$. A tripartite graph can be regarded as a weighted tripartite graph by setting the vertex weight to be $1/|V_1|$, $1/|V_2|$, $1/|V_3|$ for vertices in classes $V_1$, $V_2$, $V_3$, respectively.

Let $\alpha:=d(G[V_1,V_2])$, $\beta:=d(G[V_1,V_3])$ and $\gamma:=d(G[V_2,V_3])$. We say that $\alpha, \beta, \gamma$ is a \emph{cyclic triple} if the following conditions are satisfied:
\begin{equation*}
	\alpha\beta+\gamma>1, \ \beta\gamma+\alpha>1 \ \text{and} \ \gamma\alpha+\beta>1.
\end{equation*}
If the edge densities of $G$ form a cyclic triple, then $G$ is said to be \emph{cyclic}. Bondy, Shen, Thomass\'e and Thomassen \cite{BSTT06} proved the following theorem.
\begin{thm}[Bondy, Shen, Thomass\'e and Thomassen \cite{BSTT06}]\label{induction-base}
	If $G$ is cyclic, then it contains a triangle.
\end{thm}
For a given tripartite graph $G$ with parts $V_1,V_2,V_3$, denote its \emph{bipartite density} $d(G)$ by $$d(G):=\min_{1\leq i<j\leq 3}\frac{|E(G[V_i,V_j])|}{|V_i|\cdot |V_j|}.$$ 
Theorem \ref{induction-base} implies that any tripartite graph $G$ with $d(G)>\tau=\frac{\sqrt{5}-1}{2}$ must contain a triangle. 

In \cite{BJT10}  Baber, Johnson and Talbot investigated a tripartite version of Razborov's \cite{R2} counting result,   giving a lower bound for the number of $K_3s$ as a function of the  edge densities in a tripartite graph.  They determined the exact minimum for a range of densities and conjectured an exact bound for the remaining range. In 2021, Markstr\"om and Thomassen \cite{MT21} generalized these results to uniform hypergraphs. They determined the edge density needed to guarantee a copy of $K^{(r)}_{r+1}$, the complete $r$-uniform hypergraph on $r+1$ vertices, in an $(r+1)$-partite $r$-uniform hypergraph and  presented a sharp lower bound for the minimum number of copies of $K^{(r)}_{r+1}$ as a function of the edge densities.

In this paper we will follow up on both  Erd\H os \cite{E62} study of sets of vertex-disjoint triangles and the results on tripartite densities in \cite{BSTT06} by finding edge density conditions for the existence of $k$ vertex-disjoint triangles in a tripartite graph. 

For $1\geq\alpha,\beta,\gamma\geq 0$ and $n>k\geq 1$, we say $\alpha,\beta,\gamma$ is a \textit{$(k,n)$-cyclic triple} if inequalities (\ref{cyclic-1}), (\ref{cyclic-2}) and (\ref{cyclic-3}) hold, where:
\begin{equation}\label{cyclic-1}
	\left\{
	\begin{aligned}
		&\beta\left(\alpha-\frac{k-1}{n}\right)+\gamma>1,\\
		&\alpha\left(\beta-\frac{k-1}{n}\right)+\gamma>1.
	\end{aligned}
	\right.
\end{equation}

\begin{equation}\label{cyclic-2}
	\left\{
	\begin{aligned}
		&\gamma\left(\alpha-\frac{k-1}{n}\right)+\beta>1,\\
		&\alpha\left(\gamma-\frac{k-1}{n}\right)+\beta>1.
	\end{aligned}
	\right.
\end{equation}

\begin{equation}\label{cyclic-3}
	\left\{
	\begin{aligned}
		&\gamma\left(\beta-\frac{k-1}{n}\right)+\alpha>1,\\
		&\beta\left(\gamma-\frac{k-1}{n}\right)+\alpha>1.
	\end{aligned}
	\right.
\end{equation}

Our main result is the following theorem:
\begin{thm}\label{main-theorem}
	Let $n,k$ be integers such that $k\geq 2$ and $n\geq 5k+2$, and let $G$ be a tripartite graph with parts $A$, $B$ and $C$ such that $|A|,|B|,|C|= n$. 
	Let  $\alpha$, $\beta$ and $\gamma$ denote the edge densities of $G[A,B]$, $G[A,C]$ and $G[B,C]$, respectively. 
	
	If $\alpha, \beta,\gamma$ is a $(k,n)$-cyclic triple, then $G$ contains $k$ vertex-disjoint triangles. 
\end{thm}

\noindent\textbf{Remark:}
The conditions in (\ref{cyclic-1}), (\ref{cyclic-2}) and (\ref{cyclic-3}) are best possible. For instance, there is a tripartite graph $G$ with densities $\alpha,\beta,\gamma$ satisfying $\beta\left(\alpha-\frac{k-1}{n}\right)+\gamma=1$ and $G$ does not contain $k$ vertex-disjoint triangles. To see this, consider a tripartite graph $G$ with parts $A$, $B$ and $C$. Denote the edge densities of $G[A,B]$, $G[A,C]$ and $G[B,C]$ by $\alpha$, $\beta$ and $\gamma$, respectively. Let $B_1$ be a subset of $B$ such that $|B_1|=\frac{1-\gamma}{\beta}|B|$, $B_2$ be a subset of $ B\setminus B_1$ such that $|B_2|=(\alpha-\frac{1-\gamma}{\beta})|B|$ and $B_3=B\setminus(B_1\cup B_2)$. Let $C_1$ be a subset of $C$ such that $|C_1|=\beta |C|$ and $C_2=C\setminus C_1$. Define the edge set of $G$ to be the set of all edges between $A$ and $B_1\cup B_2$, $A$ and $C_1$, $B_2\cup B_3$ and $C$, $B_1$ and $C_2$.  

We believe that the condition on the ratio between $n$ and $k$  in Theorem \ref{main-theorem} can be improved.
\begin{conj}\label{conj1}
	The bound $n\geq 5k+2$ in Theorem \ref{main-theorem} can be improved to $n>C k$ for any constant $C>1$, for large enough $n$.
\end{conj}

Let $\tau_k(n)$ denote the positive solution of the equation $nx^2+(n+1-k)x-n=0$. Then Theorem \ref{main-theorem} implies
\begin{cor}\label{cor}
	Let $n,k$ be integers such that $k\geq 2$ and $n\geq 5k+2$, and let $G$ be a tripartite graph with parts $A,B,C$ such that $|A|,|B|,|C|=n$. If $d(G)>\tau_k(n)$, then $G$ contains $k$ vertex-disjoint triangles.
\end{cor}

For  $k=cn$  taking the  limit $n\rightarrow\infty$  we get that  $\tau_k(n)\rightarrow \tau_c=\frac{1}{2} \left(c+\sqrt{(1-c)^2+4}-1\right)$, which is equal to $\tau$ for $c=0$ and 1 for $c=1$.

Given a tripartite graph $G$  we can construct a 3-uniform hypergraph $G'$, on the same vertex set, by  taking each triangle of $G$ to be an edge in $G'$. Now $k$ vertex-disjoint triangles in $G$ correspond to a matching of size $k$ in $G'$.    The maximum number of edges in a tripartite hypergraph without a matching of size $k$ is known, see Lemma \ref{3-graph-matching},  and together with  the lower  bound on the number of triangles in a tripartite graph, see Theorem \ref{triangle-density}, we get an upper bound for the number of edges required in $G$ to have $k$ vertex-disjoint triangles. In Figure \ref{fig:bound} we plot this bound for $k=cn$ together with $\tau_c$. As we can see the hypergraph bound differs strongly from the exact bound given by $\tau_c$. This  is a consequence of the fact  that, unlike for hyperedges, the  distribution of  triangles in a graph is highly correlated. 

\begin{figure}[h]
\begin{center}
  \includegraphics[width=0.48\textwidth]{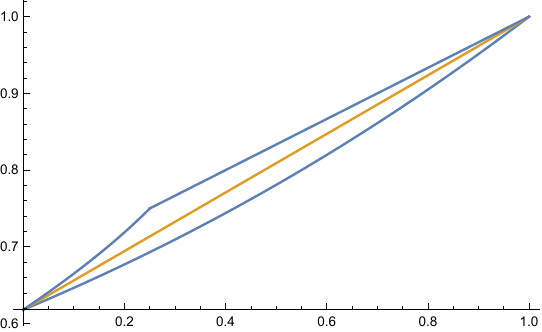}
  \end{center}
  \caption{\label{fig:bound} The hypergraph bound (top), a linear interpolation between the endpoints (middle), $\tau_c$ (bottom), as functions of  the number of triangles $c$ }
\end{figure}

Bondy, Shen, Thomass\'e and Thomassen \cite{BSTT06} also considered the problem of finding a copy of a complete graph $K_l$, with at most one vertex in each partite class, or   in a $r$-partite graph for general $l\leq r$. This line of investigation was continued for more general subgraphs by Nagy in \cite{MR2776822}, and Pfender \cite{MR2965288}, who considered triangles in multipartite graphs with more than three parts.  While the previous two papers worked with a common lower bound for all the bipartite densities  \cite{MR2942728} gave results for trees using independent densities for all bipartite subgraphs.    In \cite{MR3647823} an analogue of the  Erd{\"o}s-Stone-Simonovits theorem was proved for the case when the number of parts is large and the sought subgraph is color-critical.    More recently \cite{MR4831824}  considered many different variations on this problem, including cycles of length $l$ in $l$-partite graphs.   In our context  one may ask for  density conditions which guarantee the existence of $k$ copies of a subgraph, instead of single copy.  As a concrete example we ask:
\begin{prob}
	Let $t$ be fixed.   Assume that $G$ is a $t$-partite graph with $n$ vertices in each class and bipartite densities at least $\alpha_t(c)$. How large does $\alpha_t(c)$ have to be in order to guarantee the existence of $k=c n$ vertex disjoint copies of $K_t$ in $G$?
\end{prob}

{\ }\\

Conjecture \ref{conj1} claims that our theorem can be extended to include partial triangle-factors with $c n$ triangles for any constant $c<1$. It will not be possible to make a simple extension to include actual triangle-factors, i.e. $n=k$, since a density condition of the kind used in the theorem cannot exclude the existence of isolated vertices.   One can formulate a density condition for triangle-factors as well which, as we will now show,  becomes very restrictive.

Let $G$ be a tripartite graph with parts $V_1$, $V_2$ and $V_3$ such that $|V_1|=|V_2|=|V_3|= n$. We say that $G$ contains \emph{triangle-factor} if $G$ contains $n$ vertex-disjoint triangles. The following theorem shows that the extremal graph for triangle-factor case is different from the case in Theorem \ref{main-theorem}.
\begin{thm}\label{main-theorem2}
	Let $G$ be a tripartite graph with parts $V_1$, $V_2$ and $V_3$ such that $|V_1|=|V_2|=|V_3|= n>240$. Let $(n^2-n+\alpha'n^{\delta_a})$, $(n^2-n+\beta'n^{\delta_b})$ and $(n^2-n+\gamma'n^{\delta_c})$ be the number of edges of $G[V_1,V_2]$, $G[V_1,V_3]$ and $G[V_2,V_3]$, respectively. If 
	\begin{equation}\label{factor-edge-number}
		\left\{
		\begin{aligned}
			&\alpha'\beta' n^{(\delta_a+\delta_b-1)}+\gamma'n^{(\delta_c-1)}>1,\\
			&\gamma'\beta' n^{(\delta_c+\delta_b-1)}+\alpha'n^{(\delta_a-1)}>1,\\
			&\alpha'\gamma' n^{(\delta_a+\delta_c-1)}+\beta'n^{(\delta_b-1)}>1,
		\end{aligned}
		\right.
	\end{equation}
	then $G$ contains a triangle-factor. 
\end{thm}

\noindent\textbf{Remark:}
The condition in (\ref{factor-edge-number}) is best possible. For instance, there is a tripartite graph $G$ with densities $\alpha=1-1/n+\alpha'n^{(\delta_a-2)},\beta=1-1/n+\beta'n^{(\delta_b-2)},\gamma=1-1/n+\gamma'n^{(\delta_c-2)}$ satisfying $\alpha'\beta' n^{(\delta_a+\delta_b-1)}+\gamma'n^{(\delta_c-1)}=1$ such that  $G$ does not have a triangle-factor.
To see this, consider a tripartite graph $G$ with parts $A$, $B$ and $C$. Denote the edge densities of $G[A,B]$, $G[A,C]$ and $G[B,C]$ by $\alpha$, $\beta$ and $\gamma$, respectively. Let $x\in A$, $B_1$ be a subset of $B$ such that $|B_1|=\alpha n^2-n(n-1)$ and $C_1$ be a subset of $C$ such that $|C_1|=\beta^2-n(n-1)$. Define the edge set of $G$ to be the set of all edges between $A\setminus\{x\}$ and $B$, $A\setminus\{x\}$ and $C$, $\{x\}$ and $B_1$, $\{x\}$ and $C_1$. 

{\ }\\

The rest of the paper is organized as follows. In Section 2, we will introduce notions and calculations needed for our proofs. Section 3 is devoted to prove our main result. 
In section 4,  we will prove the triangle-factor case.

\section{Preliminaries}

Let $$R=\{(\alpha,\beta,\gamma)\in[0,1]^3:\alpha\beta+\gamma>1,\alpha\gamma+\beta>1,\beta\gamma+\alpha>1\},$$ and  $$\Delta(\alpha,\beta,\gamma)=\alpha^2+\beta^2+\gamma^2-2\alpha\beta-2\alpha\gamma-2\beta\gamma+4\alpha\beta\gamma.$$
$R_1$ and $R_2$ is a partition of $R$ such that $$R_1=\{(\alpha,\beta,\gamma\in R:\Delta(\alpha,\beta,\gamma)\geq 0)\}$$ and $R_2=R\setminus R_1$.

Let $T_{min}(\alpha,\beta,\gamma)$  denote the minimum number of triangles in tripartite graph with the stated densities, divided by $n^3$.  From \cite{BJT10},  we have the following:
\begin{thm}[ Baber, Johnson and Talbot \cite{BJT10}]\label{triangle-density}
	If $(\alpha,\beta,\gamma)\in R_1$, then $T_{min}(\alpha,\beta,\gamma)=\alpha+\beta+\gamma-2$. If $(\alpha,\beta,\gamma)\in R_2$, then $T_{min}(\alpha,\beta,\gamma)\leq 2\sqrt{\alpha\beta(1-\gamma)}+2\gamma-2$.
\end{thm}
\begin{conj}[ Baber, Johnson and Talbot \cite{BJT10}]\label{triangle-density-cnoj}
	 If $(\alpha,\beta,\gamma)\in R_2$, then $T_{min}(\alpha,\beta,\gamma) = 2\sqrt{\alpha\beta(1-\gamma)}+2\gamma-2$.
\end{conj}

For integers $n>r\geq 1$, let $[n]:=\{1,2,\cdots,n\}$ be the standard $n$ element set and $\binom{[n]}{r}=\{T\subseteq [n]:|T|=r\}$ be the collection of all its $r$-subsets. An $n$-vertex $r$-uniform hypergraph $H$ is a pair $H=(V,E)$, where $V:=[n]$ and $E(H)\subset \binom{[n]}{r}$. Let $\nu(H)$ be the \textit{matching number} of $H$, that is, the maximum number of pairwise vertex-disjoint members of $E(H)$. An $r$-uniform hypergraph $H$ is called \textit{$n$-balanced $r$-partite} if $V(H)$ is partitioned into sets $V_1,V_2,\ldots,V_r$ such that $|V_1|=|V_2|=\cdots=|V_r|=n$ and each hyperedge meets every $V_i$ in precisely one vertex. Aharoni and Howard \cite{AH17} showed the tight hyperedge density needed to guarantee the existence of a matching of size $k$ in an $n$-balanced $r$-partite $r$-uniform hypergraphs. (see Observation 1.9 in \cite{AH17})
\begin{lm}[Aharoni and Howard \cite{AH17}]\label{3-graph-matching}
	If $H$ is an $n$-balanced $r$-partite $r$-uniform hypergraph and $e(H)>(k-1)n^{r-1}$ then $\nu(H)\geq k$ and the bound is tight.
\end{lm}

For a graph $G$ and  $U\subseteq V(G)$, we use $G-U$ to denote the subgraph of $G$ induced by $V(G)\setminus U$, and we use $G[U]$ to denote the subgraph of $G$  induced by $U$. For a graph $G$ and  $E\subseteq E(G)$, we use $H-E$ to denote the graph obtained from $G$ by deleting $E$.

\begin{Def}
For a given graph $G$, let $uv$ be an edge in $G$ and $xyz$ be a triangle in $G$. We say that $uv$ \emph{sees} vertex $x$ of $xyz$ if $uvx$ is a triangle in $G$. The edge $uv$ sees $xyz$ if $uv$ sees at least one of the vertices $x$, $y$, or $z$. 
\end{Def}

For triangles $T_1$, $T_2$ in $G$, we say that $T_1$ \textit{intersects} $T_2$ if $V(T_1)\cap V(T_2)\neq \emptyset$. A set of triangles $\mathcal{T}$ is said to be \textit{intersecting} if for any $T_i, T_j\in \mathcal{T}$ we have $T_i$ intersects $T_j$.
\begin{lm}\label{intersecting-triangle}
	Let $n$ be an integer and $G$ be a tripartite graph with parts $V_1$, $V_2$ and $V_3$ such that $|V_i|\geq n>1$ for $i\in[3]$. Suppose that $G$ contains a triangle $T$ and there are no two vertex-disjoint triangles in $G$. If every vertex in $T$ seen by at most two vertex-disjoint edges disjoint from $V(T)$, then there is an edge set $E_0$ such that $E_0\cap E(T)=\emptyset$, $|E_0\cap E(G[A_i,A_{i+1}])|\leq 2$ for $i\in [3]$ and there is no triangle in $G-E_0-E(T)$. 
\end{lm}

\pf
Let $T=v_1v_2v_3$, where $v_i\in V_i$. Since there are no two vertex-disjoint triangles in $G$, every triangle in $G$ intersects $T$. If there is no vertex in $T$ seen by edges disjoint from $V(T)$, then there is no triangle in $G-E(T)$ and the proof is done. So without loss of generality we may assume that there exists at least an edge disjoint from $V(T)$ seeing $v_1$ and the number of edges seeing $v_i$ is no more than the number of edges seeing $v_1$ for $i=2,3$. Define $N_i(v_1)=\{x\in V_i\setminus {v_i}:xv_1\in E(G)\}$ for $i=2,3$. Let $M$ be a maximum matching in $G[N_2(v_1),N_3(v_1)]$. Note that $1\leq |M|\leq 2$. By K\"onig's Theorem, there is a vertex cover $V_0$ of size $|M|$ in  $G[N_2(v_1),N_3(v_1)]$. There are two cases.

\medskip
\noindent\textbf{Case 1.} $|M|=2$.

In this case, there is no edge disjoint from $V(T)$ sees $v_2, v_3$. Indeed, suppose that $v_2$ is seen by $x_1x_3$, then there exists an edge $y_2y_3\in E(G[N_2(v_1),N_3(v_1)])$ disjoint from $x_1x_3$ such that $x_1v_2x_3$ and $v_1y_2y_3$ are two vertex-disjoint triangles, a contradiction. Let $E_0=E(G[\{v_1\}\cup V_0])$. One can see that $|E_0\cap E(G[A_i,A_{i+1}])|\leq 2$ for $i\in [3]$ and there is no triangle in $G-E_0-E(T)$.

\medskip
\noindent\textbf{Case 2.} $|M|= 1$.

Define $N_i(v_2)=\{x\in V_i\setminus {v_i}:xv_2\in E(G)\}$ for $i=1,3$ and $N_i(v_3)=\{x\in V_i\setminus {v_i}:xv_3\in E(G)\}$ for $i=1,2$. Since $\nu(G[N_1(v_2),N_3(v_2)]),\nu(G[N_1(v_3),N_2(v_3)])\leq |M|=1$, there exist a vertex cover $V_0'$ of size at most one in $G[N_1(v_2),N_3(v_2)]$ and a vertex cover $V_0''$ of size at most one in $G[N_1(v_3),N_2(v_3)]$. Let $E_0=E(G[\{v_1\},V_0])\cup E(G[\{v_2\},V_0'])\cup E(G[\{v_3\},V_0''])$. Then $|E_0\cap E(G[A_i,A_{i+1}])|\leq 2$ for $i\in [3]$ and there is no triangle in $G-E_0-E(T)$. The proof is done.
\qed

If the densities $\alpha,\beta,\gamma$ are ordered by size, then the condition in the definition of $(k,n)$-cyclic triple can be simplified.
\begin{pro}\label{cyclic-condition}
	Let $1\geq\alpha, \beta\geq \gamma\geq0$ be three reals. If (\ref{cyclic-1}) holds, then (\ref{cyclic-2}) and (\ref{cyclic-3}) hold. 
\end{pro}
\pf
Since (\ref{cyclic-1}) holds, we have 
\begin{equation*}
	\begin{split}
		\gamma\left(\alpha-\frac{k-1}{n}\right)+\beta-1
		> \gamma\left(\alpha-\frac{k-1}{n}\right)+\beta-\left(\beta\left(\alpha-\frac{k-1}{n}\right)+\gamma\right)
		=(\beta-\gamma)\left(1+\frac{k-1}{n}-\alpha\right)
		\geq 0
	\end{split}
\end{equation*}
and
\begin{equation*}
	\begin{split}
		\alpha\left(\gamma-\frac{k-1}{n}\right)+\beta-1
		> \alpha\left(\gamma-\frac{k-1}{n}\right)+\beta-\left(\alpha\left(\beta-\frac{k-1}{n}\right)+\gamma\right)
		=(\beta-\gamma)\left(1-\alpha\right)
		\geq 0.
	\end{split}
\end{equation*}
Thus (\ref{cyclic-2}) holds. The proof for inequalities in (\ref{cyclic-3}) is similar.
\qed

Next we give the another form of inequalities \eqref{cyclic-1}, \eqref{cyclic-2} and \eqref{cyclic-3}, which will be used in the proof of Lemma \ref{delete2}.
\begin{equation}\label{cyclic-1'}
	\left\{
	\begin{aligned}
		&\beta\frac{\alpha n-(k-1)}{n-(k-1)}+\frac{\gamma n-(k-1)}{n-(k-1)}>1,\\
		&\alpha\frac{\beta n-(k-1)}{n-(k-1)}+\frac{\gamma n-(k-1)}{n-(k-1)}>1
	\end{aligned}
	\right.
\end{equation}

\begin{equation}\label{cyclic-2'}
	\left\{
	\begin{aligned}
		&\gamma\frac{\alpha n-(k-1)}{ n-(k-1)}+\frac{\beta n-(k-1)}{ n-(k-1)}>1,\\
		&\alpha\frac{\gamma n-(k-1)}{n-(k-1)}+\frac{\beta n-(k-1)}{ n-(k-1)}>1
	\end{aligned}
	\right.
\end{equation}

\begin{equation}\label{cyclic-3'}
	\left\{
	\begin{aligned}
		&\gamma\frac{\beta n-(k-1)}{ n-(k-1)}+\frac{\alpha n-(k-1)}{n-(k-1)}>1,\\
		&\beta\frac{\gamma n-(k-1)}{ n-(k-1)}+\frac{\alpha n-(k-1)}{n-(k-1)}>1
	\end{aligned}
	\right.
\end{equation}
\begin{pro}\label{equivalence}
$\alpha,\beta,\gamma$ is a $(k,n)$-cyclic triple if and only if inequalities \eqref{cyclic-1'}, \eqref{cyclic-2'} and \eqref{cyclic-3'} hold.
\end{pro}
\begin{proof}
	Let $\{\lambda,\eta,\mu\}:=\{\alpha,\beta,\gamma\}$. Then we only need to show $\lambda\left(\eta-\frac{k-1}{n}\right)+\mu>1$ holds if and only if $\lambda\frac{\eta n-(k-1)}{n-(k-1)}+\frac{\mu n-(k-1)}{n-(k-1)}>1$ holds. Notice that $$\lambda\frac{\eta n-(k-1)}{n-(k-1)}+\frac{\mu n-(k-1)}{n-(k-1)}>1$$ holds if and only if
	\begin{equation*}\label{equivalence-inequality}
		\lambda\left(\eta n-(k-1)\right)+\mu n-(k-1)>n-(k-1)
	\end{equation*} 
	holds and $\lambda\left(\eta-\frac{k-1}{n}\right)+\mu>1$ holds if and only if inequality \eqref{equivalence-inequality} holds. The proof is done.
\end{proof}

\section{Proof of Theorem \ref{main-theorem}}

We first prove Lemmas \ref{delete1} and \ref{delete2} to deal with vertices with large degree in $G$. 

\begin{lm}\label{delete1}
	Let $n,k$ be integers and $G$ be a tripartite graph with parts $U$, $V$ and $W$ such that  $|U|,|V|,|W|\geq n\geq k\geq 1$. Denote the edge densities of $G[U,V]$, $G[U,W]$ and $G[V,W]$ by $\lambda$, $\eta$ and $\mu$, respectively. Let $S$ be a subset of $V\cup W$  and $G'=G-S$ such that $|S|\leq k-1$. If
	\begin{equation*}
		\eta\left(\lambda-\frac{k-1}{|V|}\right)+\mu> 1 \ \text{and} \  \lambda\left(\eta-\frac{k-1}{|W|}\right)+\mu> 1,
	\end{equation*}
	then 
	\begin{equation}\label{delete1-1}
		\eta'\left(\lambda'-\frac{k-1-|S|}{|V\setminus T|}\right)+\mu'> 1 
	\end{equation}
	and
	\begin{equation}\label{delete1-2}
		\lambda'\left(\eta'-\frac{k-1-|S|}{|W\setminus T|}\right)+\mu'> 1 ,
	\end{equation}
	where $\lambda'$, $\eta'$ and $\mu'$ denote the edge densities of $G'[U,V\setminus T]$, $G'[U,W\setminus T]$ and $G'[V\setminus T,W\setminus T]$, respectively.
\end{lm}

\pf
The proof of inequalities (\ref{delete1-1}) and (\ref{delete1-2}) are similar, so we only give the proof of inequality (\ref{delete1-1}) here. 
Let $|S\cap V|=x$, $|S\cap W|=y$ and $|S|=t$. Thus
\begin{equation*}
	x+y=t.
\end{equation*} 
Let
\begin{equation*}
	p:=\eta\left(\lambda-\frac{k-1}{|V|}\right)+\mu-\left(\lambda\left(\eta-\frac{k-1}{|W|}\right)+\mu\right)=\frac{(k-1)(\lambda|V|-\eta|W|)}{|V||W|}
\end{equation*}
be the difference between $\eta\left(\lambda-\frac{k-1}{|V|}\right)+\mu$ and $\lambda\left(\eta-\frac{k-1}{|W|}\right)+\mu$.

Define $G_1=G-(S\cap V)$. For each vertex $v\in V$, the degree of $v$ is at most $|U|$ in $G[U,V]$ and at most $|W|$ in $G[V,W]$. Therefore, $e(G_1[U,V])\geq \lambda|U||V|-x|U|$ and $e(G_1[V,W])\geq \mu|V||W|-x|W|$. Denote the edge densities of $G_1[U,V]$, $G_1[U,W]$ and $G_1[V,W]$ by $\lambda_1$, $\eta_1$ and $\mu_1$, respectively. Thus
\begin{equation}\label{G_1-density}
	\lambda\geq\lambda_1\geq \frac{\lambda |U||V|-x|U|}{|U|(|V|-x)}=\frac{\lambda |V|-x}{|V|-x}, \ \mu\geq\mu_1\geq \frac{\mu |V||W|-x|W|}{|W|(|V|-x)}=\frac{\mu |V|-x}{|V|-x}, \ \eta_1=\eta.
\end{equation}

Note that $G'=G_1-(S\cap W)$, then for each vertex $w\in W$, the degree of $w$ is at most $|U|$ in $G_1[U,W]$ and at most $|V|-x$ in $G_1[V,W]$. Therefore, $e(G'[U,W])\geq \eta|U||W|-y|U|$ and $e(G'[V,W])\geq e(G_1[V,W])-y(|V|-x)=\mu_1(|V|-x)|W|-y(|V|-x)$. Thus
\begin{equation}\label{G'-density}
	\lambda'=\lambda_1, \ \mu\geq\mu'\geq \frac{\mu_1 (|V|-x)|W|-y(|V|-x)}{(|W|-y)(|V|-x)}=\frac{\mu_1 |W|-y}{|W|-y}, \ \eta\geq \eta'\geq\frac{\eta_1 |U||W|-y|U|}{(|W|-y)|U|}=\frac{\eta_1 |W|-y}{|W|-y}.
\end{equation}

By inequalities \eqref{G_1-density} and \eqref{G'-density}, in order to prove $$\eta'\left(\lambda'-\frac{k-1-t}{|V\setminus T|}\right)+\mu'> 1,$$ it suffices to prove 
$$\frac{\eta|W|-y}{|W|-y}\left(\frac{\lambda|V|-x}{|V|-x}-\frac{k-1-t}{|V|-x}\right)+\frac{\frac{(\mu |V|-x)|W|}{|V|-x}-y}{|W|-y}> 1.$$
Notice that 
\begin{equation*}
	\begin{split}
		&\frac{\eta|W|-y}{|W|-y}\left(\frac{\lambda|V|-x}{|V|-x}-\frac{k-1-t}{|V|-x}\right)+\frac{\frac{(\mu |V|-x)|W|}{|V|-x}-y}{|W|-y}-1\\
		=&\frac{(\eta|W|-y)(\lambda|V|-x-(k-1-t))}{(|W|-y)(|V|-x)}+\frac{(\mu |V|-x)|W|-y(|V|-x)}{(|W|-y)(|V|-x)}-\frac{(|W|-y)(|V|-x)}{(|W|-y)(|V|-x)}\\
		=&\frac{(\eta|W|-y)(\lambda|V|-(k-1-y))+\mu|V||W|-|V||W|}{(|W|-y)(|V|-x)}\\
		=&\frac{|V||W|((\eta-\frac{y}{|W|})(\lambda-\frac{k-1-y}{|V|})+\mu-1)}{(|W|-y)(|V|-x)}\\
		=&\frac{|V||W|(\eta(\lambda-\frac{k-1}{|V|})+\frac{\eta y}{|V|}-\frac{y}{|W|}(\lambda-\frac{k-1-y}{|V|})+\mu-1)}{(|W|-y)(|V|-x)}\\
		\geq &\frac{|V||W|(\eta(\lambda-\frac{k-1}{|V|})+\frac{y(\eta|W|-\lambda|V|)}{|V||W|}+\mu-1)}{(|W|-y)(|V|-x)}.
	\end{split}
\end{equation*}
So we only need to prove that
\begin{equation*}
	\eta\left(\lambda-\frac{k-1}{|V|}\right)+\frac{y(\eta|W|-\lambda|V|)}{|V||W|}+\mu-1>0.
\end{equation*}

There are two cases.

\medskip
\noindent\textbf{Case 1.} $p\leq 0$.

Since $p\leq 0$, $\eta|W|-\lambda|V|\geq0$. Thus $\frac{y(\eta|W|-\lambda|V|)}{|V||W|}\geq 0$. Combining with $$\eta\left(\lambda-\frac{k-1}{|V|}\right)+\mu>1,$$ the proof is done in this case.

\medskip
\noindent\textbf{Case 2.} $p>0$.

Since $$\lambda\left(\eta-\frac{k-1}{|W|}\right)+\mu> 1,$$ it suffices to prove $$\eta\left(\lambda-\frac{k-1}{|V|}\right)+\frac{y(\eta|W|-\lambda|V|)}{|V||W|}+\mu-\left(\lambda\left(\eta-\frac{k-1}{|W|}\right)+\mu\right)>0.$$
Indeed,
\begin{equation*}
	\begin{split}
		&\eta\left(\lambda-\frac{k-1}{|V|}\right)+\frac{y(\eta|W|-\lambda|V|)}{|V||W|}+\mu-\left(\lambda\left(\eta-\frac{k-1}{|W|}\right)+\mu\right)\\
		=&p-\frac{y(\lambda|V|-\eta|W|)}{|V||W|}\\
		=&\frac{(k-1)(\lambda|V|-\eta|W|)}{|V||W|}-\frac{y(\lambda|V|-\eta|W|)}{|V||W|}\\
		=&\frac{(k-1-y)(\lambda|V|-\eta|W|)}{|V||W|}\geq 0.
	\end{split}
\end{equation*}
and the proof is done in this case.
\qed

\begin{lm}\label{delete2}
	Suppose that $n, k$ are integers and $G$ is a tripartite graph with parts $A$, $B$ and $C$ such that $|A|\geq|B|,|C|\geq  n\geq k\geq 1$. Let $\alpha$, $\beta$ and $\gamma$ be the edge densities of $G[A,B]$, $G[A,C]$ and $G[B,C]$, respectively, such that $\alpha, \beta\geq \gamma$. Let $S$ be a subset of $A$  and $G'=G-S$ such that $|S|\leq k$. If 
	\begin{equation*}
			\beta\left(\alpha-\frac{k}{|B|}\right)+\gamma>1 \ \text{and} \ 
			\alpha\left(\beta-\frac{k}{|C|}\right)+\gamma>1
	\end{equation*}
	hold, then 
	\begin{equation*}
	         \beta'\left(\alpha'-\frac{k-|S|}{|B|}\right)+\gamma'>1 \ \text{and} \ 
			\alpha'\left(\beta'-\frac{k-|S|}{|C|}\right)+\gamma'>1,
	\end{equation*}
	 where $\alpha'$, $\beta'$ and $\gamma'$ is the edge densities of $G'[A\setminus T,B]$, $G'[A\setminus T,C]$ and $G'[B,C]$, respectively.
	
\end{lm}

\pf
Define $t=|S|$. Since each vertex $v\in A$ has degree at most $|B|$ in $G[A,B]$ and has degree at most $|C|$ in $G[A,C]$, we have

\begin{equation*}
	\alpha\geq\alpha'\geq \frac{\alpha |A||B|-t|B|}{|B|(|A|-t)}=\frac{\alpha |A|-t}{|A|-t}, \ \beta\geq\beta'\geq \frac{\beta |A||C|-t|C|}{|C|(|A|-t)}=\frac{\beta |A|-t}{|A|-t}, \ \gamma'=\gamma.
\end{equation*}

If $\alpha\left(\beta-\frac{k}{|C|}\right)+\gamma>1$ holds, then $\alpha\frac{\beta|C|-k}{|C|-k}+\frac{\gamma|C|-k}{|C|-k}>1$ holds by Proposition \ref{equivalence}. Again by Proposition \ref{equivalence}, in order to prove $\alpha'\left(\beta'-\frac{k-|S|}{|C|}\right)+\gamma'>1$, it suffices to prove that $\alpha'\frac{\beta'|C|-(k-t)}{|C|-(k-t)}+\frac{\gamma'|C|-(k-t)}{|C|-(k-t)}>1$. Indeed,
\begin{equation*}
	\begin{split}
		&\alpha'\frac{\beta'|C|-(k-t)}{|C|-(k-t)}+\frac{\gamma'|C|-(k-t)}{|C|-(k-t)}-1\\
		>&\frac{\alpha |A|-t}{|A|-t}\frac{\frac{\beta |A|-t}{|A|-t}|C|-(k-t)}{|C|-(k-t)}+\frac{\gamma|C|-(k-t)}{|C|-(k-t)}-\alpha\frac{\beta|C|-k}{|C|-k}-\frac{\gamma|C|-k}{|C|-k}\\
		=&\left(\alpha-\frac{1-\alpha}{|A|-t}\right)\frac{\left(\beta-\frac{1-\beta}{|A|-t}\right)|C|-(k-t)}{|C|-(k-t)}-\alpha\frac{\beta|C|-k}{|C|-k}+(1-\gamma)\left(\frac{k}{|C|-k}-\frac{k-t}{|C|-(k-t)}\right)\\
		=&\alpha\left(\frac{\beta|C|-(k-t)}{|C|-(k-t)}-\frac{\beta|C|-k}{|C|-k}\right)-\frac{\alpha(1-\beta)|C|}{(|A|-t)(|C|-(k-t))}-\frac{\beta(1-\alpha)|C|}{(|A|-t)(|C|-(k-t))}\\
		&+\frac{(1-\alpha)\left(\frac{(1-\beta)|C|}{|A|-t}+k-t\right)}{(|A|-t)(|C|-(k-t))}+(1-\gamma)\left(\frac{k}{|C|-k}-\frac{k-t}{|C|-(k-t)}\right)\\
		\geq& \left(\frac{k}{|C|-k}-\frac{k-t}{|C|-(k-t)}\right)(1-\gamma+\alpha-\alpha\beta)-\frac{|C|}{(|A|-t)(|C|-(k-t))}(\alpha+\beta-2\alpha\beta)\\
		=&\frac{|C|}{|C|-(k-t)}\left(\frac{t(1-\gamma+\alpha-\alpha\beta)}{|C|-k}-\frac{\alpha+\beta-2\alpha\beta}{|A|-t}\right)\\
		\geq&\frac{|C|}{|C|-(k-t)}\left(\frac{(1-\gamma+\alpha-\alpha\beta)-(\alpha+\beta-2\alpha\beta)}{|C|-k}\right)\\
		\geq&\frac{|C|}{|C|-(k-t)}\left(\frac{(1-\alpha)(1-\beta)}{|C|-k}\right)>0.
	\end{split}
\end{equation*}

Similarly, if $\beta\left(\alpha-\frac{k}{|B|}\right)+\gamma>1$, then we can prove that $\beta'\left(\alpha'-\frac{k-|S|}{|B|}\right)+\gamma'>1$.
\qed

Using Lemmas \ref{delete1} and \ref{delete2}, we have the following lemma.
\begin{lm}\label{delete3}
Let $n, k$ be integers and $G$ be a tripartite graph with parts $A$, $B$ and $C$ such that $|A|,|B|,|C|= n\geq k\geq 1$. Denote the edge densities of $G[A,B]$, $G[A,C]$ and $G[B,C]$ by $\alpha$, $\beta$ and $\gamma$, respectively. Let $V_0$ be a subset of $V(G)$ and $G'=G-V_0$  such that $|V_0|\leq k-1$. Suppose that $\lambda$, $\eta$ and $\mu$ are the edge densities of $G'[A\setminus V_0,B\setminus V_0]$, $G'[A\setminus V_0,C\setminus V_0]$ and $G'[B\setminus V_0,C\setminus V_0]$ such that $\lambda,\eta\geq\mu$. Let $\{U,V,W\}= \{A,B,C\}$ such that $d(G'[U\setminus V_0,V\setminus V_0])=\eta$ and $d(G'[U\setminus V_0,W\setminus V_0])=\lambda$. If $\alpha$, $\beta$, $\gamma$ is a $(k,n)$-cyclic triple,
then 
\begin{equation*}
	\lambda\left(\eta-\frac{k-1-|V_0|}{|V\setminus V_0|}\right)+\mu> 1 
\end{equation*}
and
\begin{equation*}
	\eta\left(\lambda-\frac{k-1-|V_0|}{|W\setminus V_0|}\right)+\mu> 1.
\end{equation*}

\end{lm}

\pf
Let  $G_1=G-(V_0\cap (V\cup W))$. By Lemma \ref{delete1}, we have
\begin{equation*}
	\beta_1\left(\alpha_1-\frac{k-1-|V_0\cap(V\cup W)|}{|V\setminus V_0|}\right)+\gamma_1> 1 
\end{equation*}
and
\begin{equation*}
	\alpha_1\left(\beta_1-\frac{k-1-|V_0\cap(V\cup W)|}{|W\setminus V_0|}\right)+\gamma_1> 1,
\end{equation*}
where $\alpha_1,\beta_1, \gamma_1$ denote the edge densities of $G_1[U,V\setminus V_0]$, $G_1[U,W\setminus V_0]$ and $G_1[V\setminus V_0,W\setminus V_0]$, respectively. Since $\lambda,\eta\geq\mu$ and $G'=G_1-(V_0\cap U)$, we can derive that $\alpha_1,\beta_1\geq\gamma_1$. Thus by Lemma \ref{delete2}, 
\begin{equation*}
	\lambda\left(\eta-\frac{k-1-|V_0|}{|V\setminus V_0|}\right)+\mu> 1
\end{equation*}
and
\begin{equation*}
	\eta\left(\lambda-\frac{k-1-|V_0|}{|W\setminus V_0|}\right)+\mu> 1 
\end{equation*}
hold.
\qed

{\ }\\

\noindent \textbf{Proof of Theorem \ref{main-theorem}.}
The theorem is true for 1 triangle  by Theorem \ref{induction-base}. We will next assume  that $G$ satisfies the conditions in the theorem but does not contain $k$ vertex-disjoint triangles for some $k$ strictly smaller than the bound in the theorem  and derive  a contradiction. 

Without loss of generality we can take the value of $k$ to be smallest for which the theorem fails. Let  $\mathcal{T}=\{T_1,T_2,\ldots,T_{k-1}\}$  be a maximal set of vertex-disjoint triangles in $G$.

A triangle $a_1a_2a_3$ is called \textit{good} if there are at least three vertex-disjoint edges in $E(G-V(\mathcal{T}))$ which see $a_1a_2a_3$ at the same vertex. Let $\mathcal{T}_1=\{T\in \mathcal{T}: T \ \text{is good}\}$ and $\mathcal{T}_2=\mathcal{T}\setminus \mathcal{T}_1$. There can not be two vertex-disjoint edges $e_1,e_2\in E(G-V(\mathcal{T}))$ such that $e_1$, $e_2$ see different vertices of a triangle in $\mathcal{T}$, for otherwise there exists $k$ vertex-disjoint triangles, a contradiction. Thus we can derive that there is only one vertex seen by some edge $e\in E(G-V(\mathcal{T}))$ for each triangle $T\in \mathcal{T}_1$. Denote the vertex seen by some edge in $E(G-V(\mathcal{T}))$ by $v_i$ for each $T_i\in \mathcal{T}_1$. 
Let $V_1=\{v_i:T_i\in \mathcal{T}_1\}$ and define $|V_1|=|\mathcal{T}_1|=t$. 

Let $U_1=A\setminus V_1$, $U_2=B\setminus V_1$ and $U_3=C\setminus V_1$.
Define $G'=G-V_1$ and denote the edge densities of $G'[U_1,U_2]$, $G'[U_2,U_3]$ and $G'[U_1,U_3]$ by $\alpha'$, $\beta'$ and $\gamma'$, respectively. Without loss of generality, we may assume that $\alpha'\geq\beta'\geq \gamma'$. 
Thus by Lemma \ref{delete3}, 
\begin{equation*}
	\alpha'\left(\beta'-\frac{k-1-t}{|U_3|}\right)+\gamma'> 1 
\end{equation*}
and
\begin{equation*}
	\beta'\left(\alpha'-\frac{k-1-t}{|U_2|}\right)+\gamma'> 1
\end{equation*}
hold.  If $t=k-1$, then by Theorem \ref{induction-base} and Proposition \ref{cyclic-condition}, $G'$ has a triangle, a contradiction. So we assume that $t\leq k-2$ in the following proof.

Let $\mathcal{S}$ denote the set of all triangles in $G'$.  For every $e\in E(G')$, let $\mathcal{T}(e)=\{T\in \mathcal{S}: e\in E(T)\}$.  Next we will find an edge set $E'$ satisfying
\begin{equation}\label{delete-edge-set}
\mathcal{S}= \cup_{e\in E'}\mathcal{T}(e). 
\end{equation}
We will show that the edge density of $G'-E'$ is large enough to guarantee the existence of a triangle, but $G'-E'$ contains no triangle according to \eqref{delete-edge-set}, which is a contradiction. 

Define $V_2=V(\mathcal{T}_1)\setminus V_1$. Since the maximum number of vertex-disjoint triangles in $G$ is $k-1$, every triangle in $G$ intersects a triangle in $\mathcal{T}$. Thus every triangle in $G'$ has one vertex contained in $V_2\cup V(\mathcal{T}_2)$. Let $R=V(G')\setminus (V_2\cup V(\mathcal{T}_2))$. We now split the set $\mathcal{S}$ into there parts. Let $\mathcal{S}_1=\{T\in \mathcal{S}:V(T)\cap V_2=\emptyset\}\cup\{T\in \mathcal{S}: |V(T)\cap V_2|=1,|V(T)\cap V(\mathcal{T}_2)=2|\}$, $\mathcal{S}_2=\{T\in \mathcal{S}: V(T)\cap V(\mathcal{T}_2)=\emptyset\}$, $\mathcal{S}_3=\{T\in \mathcal{S}:|V(T)\cap V_2|=1, |V(T)\cap V(\mathcal{T}_2)|=1, |V(T)\cap R|=1 \}\cup\{T\in \mathcal{S}:|V(T)\cap V_2|=2, |V(T)\cap V(\mathcal{T}_2)|=1 \}$. One can see that $\mathcal{S}=\mathcal{S}_1\cup \mathcal{S}_2\cup \mathcal{S}_3$.

Given $T_i\in \mathcal{T}$, define $\mathcal{S}_{T_i}=\{T\in \mathcal{S}:T\cap T_i\neq \emptyset\}$ and $\mathcal{S}_{T_i}'=\mathcal{S}_{T_i}\setminus \left(\sum_{j\neq i,j\in [k-1]}\mathcal{S}_{T_j}\right)$.  $\mathcal{S}_{T_i}'$ is a set of intersecting triangles, otherwise there are two vertex-disjoint triangles $T',T''\in \mathcal{S}_{T_i}'$ such that $\{T',T''\}\cup \mathcal{T}\setminus \{T_i\}$ is a set of $k$ vertex-disjoint triangles, a contradiction. Recall that for each triangle $T\in \mathcal{T}_2$ and each vertex $x\in V(T)$, there are at most two vertex-disjoint edges in $E(G'-(V_2\cup V(\mathcal{T}_2)))$ seeing $T$ at $x$. By Lemma \ref{intersecting-triangle}, there is an edge set $E_{T_i}$ such that $|E_{T_i}\cap E(G'[U_1,U_2])|,|E_{T_i}\cap E(G'[U_2,U_3])|,|E_{T_i}\cap E(G'[U_1,U_3])|\leq 2$ such that there is no triangle in $G'-E_{T_i}-E(T_i)$ belonging to $\mathcal{S}_{T_i}'$ for each $T_i\in \mathcal{T}_2$. That is $\cup _{T_i\in \mathcal{T}_2}\mathcal{S}_{T_i}'=\cup_{e\in (\cup_{T_i\in \mathcal{T}_2}E_{T_i})}\mathcal{T}(e)$. Since every triangle in $\mathcal{S}_1\setminus \cup _{T_i\in \mathcal{T}_2}\mathcal{S}_{T_i}'$ has an edge contained in $E(G'[V(\mathcal{T}_2)])$, we have 
\begin{equation}\label{s1}
\mathcal{S}_1\subseteq \cup_{e\in Q_1}\mathcal{T}(e),
\end{equation}
where $Q_1=(\cup_{T_i\in \mathcal{T}_2}E_{T_i})\cup E(G'[V(\mathcal{T}_2)])$.
For $i\in [3]$, we have 
\begin{equation}\label{q1}
\begin{split}
|Q_1\cap E(G'[U_i,U_{i+1}])|&\leq|E(G'[V(\mathcal{T}_2)])\cap E(G'[U_i,U_{i+1}])|+\sum_{T_i\in \mathcal{T}_2}|E_{T_i}\cap E(G'[U_i,U_{i+1}])|\\
&\leq(k-1-t)^2+2(k-1-t).
\end{split}
\end{equation}

\begin{cla}\label{s2}
	$\mathcal{S}_2=\emptyset$.
\end{cla}
\pf
Suppose that there exists a triangle $T$ such that $T\cap V(\mathcal{T}_2)=\emptyset$. Let $\mathcal{I}=\{T_i\in \mathcal{T}_1:T_i\cap T\neq \emptyset\}$. By the definition of $\mathcal{T}_1$,  every vertex $v\in V_1\cap V(\mathcal{I})$ is seen by a matching in $G'-(V_2\cup V(\mathcal{T}_2))$ of size three. Thus there is a set of vertex-disjoint triangles $\mathcal{T}'$ of size $|\mathcal{I}|$ such that $(\mathcal{T}\setminus \mathcal{I})\cup \mathcal{T}'\cup\{T\}$ is a set of $k$ vertex-disjoint triangles in $G$, a contradiction.
\qed

Let $\mathcal{T}_1=\{T_1,T_2,\ldots,T_t\}$ and $\mathcal{T}_2=\{T_{t+1},T_{t+2},\ldots,T_{k-1}\}$. For any $T_i\in \mathcal{T}_2$ and $T_j\in \mathcal{T}_1$,  define $T_i=u^i_{j_1}u^i_{j_2}u^i_{j_3}$ and $T_j=x^j_1x^j_2x^j_3$ such that $x^j_1\in V_1$. 
Without loss of generality, we may assume that $x^j_\ell,u^i_{j_\ell}$ in the same part of $G'$ for $\ell=1,2,3$. Let $G_{ij}=G'[V(T_i)\cup (V(T_j)\setminus V_1)]$, $\mathcal{S}_{ij}=\{T\in \mathcal{S}_3: E(T)\cap E(G_{ij})\neq \emptyset\}$ and $R_{ij}=\{u^i_{j_1}x^j_2, u^i_{j_1}x^j_3\}$.

\begin{cla}\label{s3}
   For any $i\in [t]$ and $j\in[k-1]\setminus [t]$, there is an edge set $E_{ij}$ such that $|(E_{ij}\cup R_{ij})\cap E(G[U_\ell,U_{\ell+1}])|\leq 1$ for $\ell\in[3]$ and $\mathcal{S}_{3}\subseteq \cup_{e\in Q_2}\mathcal{T}(e)$, where $Q_2=\cup_{i=1}^t\cup _{j=t+1}^{k-1}(E_{ij}\cup R_{ij})$.
\end{cla}
\pf
For each $G_{ij}$, there are four cases.

\medskip
\noindent\textbf{Case 1.}  One  of $x^j_2u^i_{j_3}, x^j_3u^i_{j_2}$ does not belong to a triangle in $\mathcal{S}_{3}$.

In this case, it is not difficulty to see that $\{x^j_2u^i_{j_3},x^j_3u^i_{j_2},x^j_3u^i_{j_1},x^j_2u^i_{j_1}\}\cap E(G_{ij})$ is the desired set $E_{ij}$ satisfying $|(E_{ij}\cup R_{ij})\cap E(G[U_\ell,U_{\ell+1}])|\leq 1$ for $\ell\in[3]$ and $\mathcal{S}_{ij}\subseteq \cup_{e\in E_{ij}\cup R_{ij}}\mathcal{T}(e)$.

\medskip
\noindent\textbf{Case 2.}  There exist distinct vertices $v_1$ and $w_1$ such that $v_1x^j_2u^i_{j_3}, w_1x^j_{3}u^i_{j_2}$ are triangles in $\mathcal{S}_{3}$.

Let $\mathcal{F}=\{T\in \mathcal{T}_1:V(T)\cap \{v_1,w_1,x^j_1\}\neq \emptyset\}$ and $V'=\{x\in V_1:x\in V(\mathcal{F})\}$. For each $x\in V'$, there are three vertex-disjoint edges in $E(G'-(V_2\cup V(\mathcal{T}_2)))$ see $x$, thus there is a set of vertex-disjoint triangles $\mathcal{F}'$ of size $|V'|$ such that  $V(\mathcal{F}')\cap (V(\mathcal{T})\cup \{v_1,w_1\})=V'$. One can see that $\mathcal{T}\setminus(\{u^i_{j_1}u^i_{j_2}u^i_{j_3}\}\cup \mathcal{F})\cup \{v_1x^j_2u^i_{j_3},w_1x^j_3u^i_{j_2}\}\cup \mathcal{F}'$ is a set of vertex-disjoint triangles of size $k$ in $G$, a contradiction.

\medskip
\noindent\textbf{Case 3.} There exists a vertex $v\notin V(\mathcal{T})$ such that $vx^j_2u^i_{j_3}$ is the only triangle in $\mathcal{S}_3$ contains $x^j_2u^i_{j_3}$ and $vx^j_3u^i_{j_2}$ is the only triangle in $\mathcal{S}_3$ contains $x^j_3u^i_{j_2}$.

If $|\{u^i_{j_1}x^j_2,u^i_{j_1}x^j_3\}\cap E(G_{ij})|\geq 1$, then assume that $u^i_{j_1}x^j_2\in E(G')$ without loss of generality. Since $x^j_1\in V_1$, there is a triangle $x^j_1z_2z_3$ such that $z_2,z_3\notin V(\mathcal{T})$. Thus $(\mathcal{T}\setminus \{x^j_1x^j_2x^j_3,u^i_{j_1}u^i_{j_2}u^i_{j_3}\})\cup \{vu^i_{j_2}x^j_3,u^i_{j_1}x^j_2u^i_{j_3}, x^j_1z_2z_3\}$ is a set of vertex-disjoint triangles of size $k$ in $G$, a contradiction. Thus we have $|\{u^i_{j_1}x^j_2,u^i_{j_1}x^j_3\}\cap E(G_{ij})|=0$.
Let $E_{ij}=\{vx^j_2, vx^j_3\}$. One can see that $E_{ij}$ is the desired set  satisfying $|(E_{ij}\cup R_{ij})\cap E(G[U_\ell,U_{\ell+1}])|\leq 1$ for $\ell\in[3]$ and $\mathcal{S}_{ij}\subseteq \cup_{e\in E_{ij}\cup R_{ij}}\mathcal{T}(e)$.

\medskip
\noindent\textbf{Case 4.} There exists a vertex $v\in V(\mathcal{T})$ such that $vx^j_2u^i_{j_3}$ is the only triangle in $\mathcal{S}_3$ contains $x^j_2u^i_{j_3}$ and $vx^j_3u^i_{j_2}$ is the only triangle in $\mathcal{S}_3$ contains $x^j_3u^i_{j_2}$.

Since $vx^j_2u^i_{j_3}\in \mathcal{S}_3$, we have $v\in V(\mathcal{T}_1)$. It follows that there is a triangle $T_p=x^p_1x^p_2x^p_3\in \mathcal{T}_1$ such that $v\in \{x^p_2, x^p_3\}$ since $x^p_1\in V_1$. Without loss of generality, we may assume that $x^p_2$ and $x^j_2$ are in the same part of $U_1,U_2,U_3$. We can derive that $v=x^p_3$ and $u^i_{j_3}=u^i_{p_1}$.  Thus $u^i_{j_3}v=u^i_{p_1}x^p_3\in R_{ip}$. Let $E_{ij}=\{u^i_{j_2}x^j_{3},u^i_{j_1}x^j_2,u^i_{j_1}x^j_3\}\cap E(G_{ij})$. One can see that $E_{ij}$ is the desired set  satisfying $|(E_{ij}\cup R_{ij})\cap E(G[U_\ell,U_{\ell+1}])|\leq 1$ for $\ell\in[3]$ and $\mathcal{S}_{ij}\subseteq  \cup_{e\in E_{ij}\cup R_{ip}}\mathcal{T}(e)$.

Note that $\mathcal{S}_3=\cup_{i=1}^t\cup _{j=t+1}^{k-1}\mathcal{S}_{ij}$. Thus $\mathcal{S}_3\subseteq \cup_{e\in Q_2}\mathcal{T}(e)$, where $Q_2=\cup_{i=1}^t\cup _{j=t+1}^{k-1}(E_{ij}\cup R_{ij})$.
\qed

For $i\in [3]$, we have $$|Q_2\cap E(G'[U_i,U_{i+1}])|\leq \sum_{i=1}^t\sum _{j=t+1}^{k-1}|(E_{ij}\cup R_{ij})\cap E(G[U_\ell,U_{\ell+1}])| \leq t(k-1-t).$$
Combining with inequality \eqref{q1}, we have
\begin{equation*}
	\begin{split}
		|(Q_1\cup Q_2)\cap E(G'[U_i,U_{i+1}])|\leq (k-1-t)^2+t(k-1-t)+2(k-1-t)=(k-1-t)(k+1).
	\end{split}
\end{equation*}
By Claim \ref{s2}, Claim \ref{s3} and \eqref{s1}, we have $$\mathcal{S}=\cup_{e\in Q_1\cup Q_2}\mathcal{T}(e).$$ 
Thus $G'-Q_1-Q_2$ has no triangle. Let $G''=G'-Q_1-Q_2$. For any $i\in [3]$, the edge density of $G''[U_i,U_{i+1}]$ is at least $\rho-\frac{(k-1-t)(k+1)}{(|U_i|)(|U_{i+1}|)}$, where $\rho\in \{\alpha',\beta',\gamma'\}$. But
\begin{equation*}
	\begin{split}
	&\left(\alpha'-\frac{(k-1-t)(k+1)}{(|U_1|)(|U_2|)}\right)\left(\beta'-\frac{(k-1-t)(k+1)}{(|U_1|)(|U_3|)}\right)+\gamma'-\frac{(k-1-t)(k+1)}{(|U_2|)(|U_3|)}-1\\
		>&\left(\alpha'-\frac{(k-1-t)(k+1)}{n(n-t)}\right)\left(\beta'-\frac{(k-1-t)(k+1)}{n(n-t)}\right)+\gamma'-\frac{(k-1-t)(k+1)}{n(n-t)}-\alpha'\left(\beta'-\frac{k-1-t}{n}\right)-\gamma'\\
		>&\frac{\alpha'(k-1-t)}{n}-\frac{(\alpha'+\beta')((k-1-t)(k+1))}{n(n-t)}-\frac{(k-1-t)(k+1)}{n(n-t)}\\
		\geq&\frac{k-1-t}{n}\left(\alpha'-\frac{(\alpha'+\beta'+1)(k+1)}{n-(k-2)}\right)\\
		\geq&\frac{k-1-t}{n(n-(k-2))}\left(\alpha'(n-(k-2))-(2\alpha'+1)(k+1)\right)\geq 0.
	\end{split}
\end{equation*}
Here the last inequality follows from $\alpha'>\tau>1/2$ together with  the assumption that $n\geq 5k+2$.
By Theorem \ref{induction-base} and Proposition \ref{cyclic-condition}, $G''$ has a triangle, a contradiction.
\qed

\section{Triangle-factors}
In this section we study the triangle-factor case. 
For a tripartite graph $G$ with parts $V_1,V_2,V_3$ and a vertex $x\in V(G)$, let $d_{ij}(x):=|\{e\in E(G[V_i,V_j]):x\in e\}|$ and $\delta(G[V_i,V_j]):=\min_{v\in V_i\cup V_j}d_{ij}(v)$.

\begin{pro}\label{density=1}
	Let $G$ be a tripartite graph with parts $V_1, V_2, V_3$ such that $|V_1|=|V_2|=|V_3|= n$ and $G[V_1,V_2]$ is a complete bipartite graph. If $G[V_2,V_3]$ and $G[V_1,V_3]$ both have a perfect matching, then $G$ contains a triangle-factor. 
\end{pro}
\pf
Denote the vertices in $V_3$ by $\{v_1,v_2,\ldots,v_n\}$. Since $G[V_2,V_3]$ and $G[V_1,V_3]$ both have a perfect matching, there exist a perfect matching $\{e_1,e_2,\ldots,e_n\}$ in $G[V_2,V_3]$ and a perfect matching $\{f_1,f_2,\ldots,f_n\}$ in $G[V_1,V_3]$ such that $e_i\cap f_i=\{v_i\}$ for each $i\in [n]$. For every $i$, there is a triangle contain edges $e_i$ and $f_i$ since $G[V_1,V_2]$ is a complete bipartite graph. So $G$ has a triangle-factor. 
\qed

Now we prove Theorem \ref{main-theorem2}. Actually we prove the following theorem. One can see that condition \eqref{factor-edge-number} is equivalent to condition \eqref{factor-density} by setting $\alpha=1-1/n+\alpha'n^{(\delta_a-2)}$, $\beta=1-1/n+\beta'n^{(\delta_b-2)}$ and $\gamma=1-1/n+\gamma'n^{(\delta_c-2)}$.
\begin{thm}
	Let $G$ be a tripartite graph with parts $V_1$, $V_2$ and $V_3$ such that $|V_1|=|V_2|=|V_3|= n>240$. Let $\alpha$, $\beta$ and $\gamma$ be the edge densities of $G[V_1,V_2]$, $G[V_1,V_3]$ and $G[V_2,V_3]$, respectively. If 
	\begin{equation}\label{factor-density}
		\left\{
		\begin{aligned}
			&(\alpha n-(n-1))(\beta n-(n-1))+\gamma>1,\\
			&(\alpha n-(n-1))(\gamma n-(n-1))+\beta>1,\\
			&(\beta n-(n-1))(\gamma n-(n-1))+\alpha>1,
		\end{aligned}
		\right.
	\end{equation}
	then $G$ contains a triangle-factor. 
\end{thm}

\pf
By inequalities \eqref{factor-density}, one can see that 
\begin{equation}\label{factor-density2}
\rho\geq\frac{n(n-1)+1}{n^2} 
\end{equation}
for $\rho\in \{\alpha,\beta,\gamma\}$.

By inequality \eqref{factor-density2}, we have $e(G[V_i,V_{i+1}])\geq n(n-1)+1$ for $i\in [3]$. Thus $G[V_i,V_{i+1}]$ has a perfect matching for $i\in [3]$ by Lemma \ref{3-graph-matching}. Suppose that there is only one bipartite graph satisfying $e(G[V_i,V_{i+1}])>n(n+1)+1$. Without loss of generality, assume $i=1$. Then $e(G[V_1,V_3])=e(G[V_2,V_3]) = n(n-1)+1 $. It follows that $\gamma=\beta= \frac{n(n-1)+1}{n^2}$. By inequalities \eqref{factor-density}, $\alpha>1-\frac{1}{n^2}$. Therefore $e(G[V_1,V_2])>n^2-1$. That is, $G[V_1,V_2]$ is a complete bipartite graph. Thus $G$ contains a triangle-factor by Proposition \ref{density=1} and the proof is done. 
So we assume that there are at least two bipartite graphs satisfying $e(G[V_i,V_{i+1}])\geq n(n-1)+2 $ in the following proof. Since the degree of any vertex is at most $n$, we can derive that 
\begin{equation}\label{minimum-degree}
\delta(G[V_i,V_{i+1}])\geq 2
\end{equation}
for $i=1,3$ without loss of generality.

Let $S_{12}=\{v\in V_1\cup V_2: d_{12}(v)<  n/2\}$, $S_{23}=\{v\in V_2\cup V_3: d_{23}(v)<  n/2\}$ and $S_{13}=\{v\in V_1\cup V_3: d_{13}(v)<  n/2\}$. By inequality (\ref{factor-density2}), we have $|S_{12}|,|S_{23}|,|S_{13}|\leq 1$. 
\begin{cla}\label{triangle-cover}
	For each vertex $v\in V(G)$, there is a triangle containing $v$.
\end{cla}
\pf
Without loss of generality, we may assume that $v\in V_1$. Let $S=V_1\setminus \{v\}$ and $G_1=G-S$. Let $\alpha',\beta',\gamma'$ be the density of $G_1[\{v\},V_2]$, $G_1[\{v\},V_3]$ and $G_1[V_2,V_3]$, respectively. Thus $\alpha'\geq \frac{\alpha n^2-n(n-1)}{n}$, $\beta'\geq \frac{\beta n^2-n(n-1)}{n}$ and $\gamma'=\gamma$. By inequality (\ref{factor-density}), we can derive that $\alpha'\beta'+\gamma'>1$. It follows that $\alpha',\beta'>0$. That is, $|N_2(v)|,| N_3(v)|\geq 1$, where $N_2(v), N_3(v)$ denote the neighborhood of $v$ in $V_2$ and $V_3$, respectively.

In order to prove that $G_1$ has a triangle, we only need to show that there exists an edge in $G_1[N_2(v),N_3(v)]$. Suppose that $E(G_1[N_2(v),N_3(v)])=\emptyset$, then $\gamma'\leq \frac{n^2-|N_2(v)|\cdot |N_3(v)|}{n^2}$. Since $\alpha'=\frac{|N_2(v)|}{n}$, $\beta'=\frac{|N_3(v)|}{n}$, we can derive that $\alpha'\beta'+\gamma'\leq \frac{|N_2(v)|\cdot |N_3(v)|}{n^2}+\frac{n^2-|N_2(v)|\cdot |N_3(v)|}{n^2}=1$, a contradiction.
\qed

\begin{cla}\label{cover-small-degree-vertex}
	There are at most $3$ vertex-disjoint triangles containing all vertices in $S_{12}\cup S_{13}\cup S_{23}$.
\end{cla}
\pf
For the case $|S_{12}\cup S_{13}\cup S_{23}|\leq 1$, The conclusion holds by Claim \ref{triangle-cover}. So we assume $2\leq |S_{12}\cup S_{13}\cup S_{23}|\leq 3$ in the following proof. 
Suppose that there is a vertex $x$ such that $x$ has degree one in some $G[V_i,V_{i+1}]$. Then $x\in S_{23}$ by inequality \eqref{minimum-degree}. Without loss of generality, we assume that there exists a vertex $y\in S_{13}\setminus\{x\}$. That is $d_{13}(y)< n/2$. Thus
\begin{equation} 
	\alpha>1-(\gamma n-(n-1))(\beta n-(n-1))\geq1-\frac{1}{2n}.
\end{equation}
It follows that $e(G[V_1,V_2])> n^2-n/2$. Thus $S_{12}=\emptyset$ and  $|S_{12}\cup S_{13}\cup S_{23}|=2$.
By Claim \ref{triangle-cover} there is a triangle $T_1$ containing $x$. If $y\in T_1$, then the proof is done, so we may assume that $y\notin T_1$. Since $d_{13}(y)\geq 2$, there is an edge $yz\in E(G[V_1,V_3])$  such that $V(T_1)\cap \{y,z\}=\emptyset$. Let $\{y,z\}= \{y_1,y_3\}$ such that $y_i\in V_i$ for $i=1,3$. Since $e(G[V_2,V_3])=n(n-1)+1$ and $d_{23}(x)=1$, one can see that $d_{23}(y_3)=n$. Note that $d_{12}(y_1)\geq 2$. Thus $|N_2(y_1)\cap N_2(y_3)|\geq d_{12}(y_1)+d_{23}(y_3)-n\geq 2$, where $N_2(y_1), N_2(y_3)$ denote the neighborhood of $y_1,y_3$ in $V_2$. So there is a triangle disjoint from $T_1$ containing $y$ and the proof is done. So from now on we assume that $\delta(G[V_i,V_{i+1}])\geq 2$ for $i\in [3]$. 

\medskip
\noindent\textbf{Case 1.} $|S_{12}\cup S_{13}\cup S_{23}|=2$

Since $|S_{12}|, |S_{13}|, |S_{23}|\leq 1$, there exists a vertex $y$ belongs to only one of $S_{12},S_{13},S_{23}$. Let $\{x,y\}= S_{12}\cup S_{13}\cup S_{23}$. Without loss of generality, we assume that $y\in S_{23}\setminus (S_{12}\cup S_{13})$ and $x\in S_{12}$. 
By Claim \ref{triangle-cover} there is a triangle $T_1$ containing $x$. If $y\in T_1$, then the proof is done, so we may assume that $y\notin T_1$. Since $d_{23}(y)\geq 2$, there is an edge $yz\in E(G[V_2,V_3])$ such that $V(T_1)\cap \{y,z\}=\emptyset$. Let $\{y,z\}= \{y_2,y_3\}$ such that $y_i\in V_i$ for $i=2,3$. Since $d_{12}(x)< n/2$, $d_{12}(y_2)> n/2+1$. Recall that $y,z\notin S_{13}$, thus $d_{13}(y_3)\geq n/2$. Therefore $|N_1(y_2)\cap N_1(y_3)|\geq d_{12}(y_2)+d_{13}(y_3)-n\geq 2$. So there is a triangle disjoint from $T_1$ containing $y$ and the proof is done. 

\medskip
\noindent\textbf{Case 2.} $|S_{12}\cup S_{13}\cup S_{23}|=3$.

Let $x\in S_{12}$, $y\in S_{23}$ and $z\in S_{13}$. Without loss of generality, we assume that $d_{12}(x)\leq d_{23}(y)\leq d_{13}(z)$. We claim that $d_{13}(z)\geq 3$. Otherwise if $d_{12}(x)\leq d_{23}(y)\leq d_{13}(z)\leq 2$, then
$$(\alpha n-(n-1))(\beta n-(n-1))+\gamma<1,$$
a contradiction. 
By Claim \ref{triangle-cover} there is a triangle $T_1$ containing $x$. If $y,z\in T_1$, then the proof is done, so we may assume that $y\notin T_1$. Since $d_{23}(y)\geq 2$, there is an edge $yu\in E(G[V_2,V_3])$ such that $V(T_1)\cap \{y,u\}=\emptyset$. Let $\{y,u\}= \{y_2,y_3\}$ such that $y_i\in V_i$ for $i=2,3$. Since $d_{12}(x) < n/2$, $d_{12}(y_2)> n/2+1$. Therefore $|N_1(y_2)\cap N_1(y_3)|\geq d_{12}(y_2)+d_{13}(y_3)-n\geq 2$ since $d_{13}(y_3)\geq n/2$. So there is a triangle $T_2$ disjoint from $T_1$ containing $y$. If $z\in T_1\cup T_2$, then the proof is done, so we may assume that $z\notin T_1\cup T_2$. Since $d_{13}(z)\geq 3$, there is an edge $zw\in E(G[V_1,V_3])$ such that $(V(T_1)\cup V(T_2))\cap \{z,w\}=\emptyset$. Let $\{z,w\}= \{z_1,z_3\}$ such that $z_i\in V_i$ for $i=1,3$. Since $d_{12}(x) < n/2$ and $d_{23}(y) < n/2$, we have $d_{12}(z_1)> n/2+1$ and $d_{23}(z_3)> n/2+1$. Thus $|N_2(z_1)\cap N_2(z_3)|\geq d_{12}(z_1)+d_{23}(z_3)-n>2$. So there is a triangle $T_3$ disjoint from $T_1\cup T_2$ containing $z$.
\qed

By Claim \ref{cover-small-degree-vertex}, there is a  set of vertex-disjoint triangles $\mathcal{T}_1$ such that $S_{12}\cup S_{23}\cup S_{13}\subseteq V(\mathcal{T}_1)$ and $|\mathcal{T}_1|\leq 3$.
Let $S'_{12}=\{v\in V_1\cup V_2: d_{12}(v)\leq 4n/5\}$, $S'_{23}=\{v\in V_2\cup V_3: d_{23}(v)\leq 4n/5\}$ and $S'_{13}=\{v\in V_1\cup V_3: d_{13}(v)\leq 4n/5\}$. Note that $|S'_{12}|,|S'_{13}|,|S'_{23}|\leq 4$. For any vertex $v\in (S'_{12}\cup S'_{23}\cup S'_{13})\setminus (S_{12}\cup S_{23}\cup S_{13})$, suppose that $v\in V_1$, there are at least $24n+1$ triangles containing $v$ since $d_{12}(v),d_{13}(v)\geq n/2$ and $e(G[N_2(v),N_3(v)])>d_{12}(v)\cdot d_{13}(v)-n>24n$. Let $\mathcal{T}_2=\{T_1,T_2,\ldots,T_{\ell}\}$ be a maximum set of vertex-disjoint triangles disjoint from $V(\mathcal{T}_1)$ such that each $T_i$ containing exactly one vertex in $(S'_{12}\cup S'_{23}\cup S'_{13})\setminus (S_{12}\cup S_{23}\cup S_{13})$. Note that $\ell+|\mathcal{T}_1|\leq 12$ and there are at most $24n$ triangles containing $v$ and intersecting some triangle in $\mathcal{T}_1\cup \mathcal{T}_2$ for any vertex $v\in (S'_{12}\cup S'_{23}\cup S'_{13})\setminus (S_{12}\cup S_{23}\cup S_{13})$. We claim that $(S'_{12}\cup S'_{23}\cup S'_{13})\setminus (S_{12}\cup S_{23}\cup S_{13})\subseteq V(\mathcal{T}_2)$, otherwise there is a vertex $u\in (S'_{12}\cup S'_{23}\cup S'_{13})\setminus (S_{12}\cup S_{23}\cup S_{13})$ such that $u\notin V(\mathcal{T}_2)$, but there exists a triangle $T'$ containing $u$ such that $V(T')\cap V(\mathcal{T}_1\cup\mathcal{T}_2)=\emptyset$, a contradiction. 

Let $G'=G-V(\mathcal{T}_1\cup\mathcal{T}_2)$. Since $n>240$, every vertices in $G'$ has degree at least $4n/5-12>3n/4$. Thus there is a perfect matching $M$ in $G[V_1,V_2]$ by Hall's Theorem. Let $M=\{u_1v_1,u_2v_2,\ldots,u_nv_n\}$. Define $H$ to be a bipartite graph with vertex parts $M$  and $V_3$. Define $E(H)=\{ex:e\in M, x\in V_3 \ \text{and} \ ev \ \text{is a triangle}\}$. For $uv\in M$, since $\delta(G')> 3n/4$, we have $d_{H}(e)>3n/4+3n/4-n=n/2 $. Similarly, we can derive that $d_{H}(x)>n/2$ for $x\in V_3$.  By Hall's Theorem, there is a  perfect matching in $H$. That is, there is a triangle-factor in $G'$. Thus $G$ contains a triangle-factor.
\qed

\bibliographystyle{amsalpha}

\begin{thebibliography}{FRMZ21}

\bibitem[ABHP15]{ABHP15}
Peter Allen, Julia B{\"o}ttcher, Jan Hladk{\`y}, and Diana Piguet, \emph{A
  density {C}orr{\'a}di--{H}ajnal theorem}, Canad. J. Math. \textbf{67} (2015),
  no.~4, 721--758.

\bibitem[AH17]{AH17}
Ron Aharoni and David Howard, \emph{A rainbow {$r$}-partite version of the
  {E}rd{\H o}s-{K}o-{R}ado theorem}, Combin. Probab. Comput. \textbf{26}
  (2017), no.~3, 321--337. \MR{3628907}

\bibitem[BFRS24]{MR4831824}
Leila Badakhshian, Victor Falgas-Ravry, and Maryam Sharifzadeh, \emph{On
  density conditions for transversal trees in multipartite graphs}, Electron.
  J. Combin. \textbf{31} (2024), no.~4, Paper No. 4.51, 27. \MR{4831824}

\bibitem[BJT10]{BJT10}
Rahil Baber, J.~Robert Johnson, and John Talbot, \emph{The minimal density of
  triangles in tripartite graphs}, LMS J. Comput. Math. \textbf{13} (2010),
  388--413. \MR{2685132 (2011g:05254)}

\bibitem[BSTT06]{BSTT06}
Adrian Bondy, Jian Shen, St{\'e}phan Thomass{\'e}, and Carsten Thomassen,
  \emph{Density conditions for triangles in multipartite graphs}, Combinatorica
  \textbf{26} (2006), no.~2, 121--131. \MR{2223630 (2007a:05062)}

\bibitem[CN12]{MR2942728}
P\'eter Csikv\'ari and Zolt\'an~L\'or\'ant Nagy, \emph{The density {T}ur\'an
  problem}, Combin. Probab. Comput. \textbf{21} (2012), no.~4, 531--553.
  \MR{2942728}

\bibitem[Erd55]{MR81469}
Paul Erd{\H o}s, \emph{Some theorems on graphs}, Riveon Lematematika \textbf{9}
  (1955), 13--17. \MR{81469}

\bibitem[Erd62]{E62}
\bysame, \emph{{\"U}ber ein extremalproblem in der graphentheorie}, Archiv der
  Mathematik \textbf{13} (1962), no.~1, 222--227.

\bibitem[FRMZ21]{FMZ21}
Victor Falgas-Ravry, Klas Markstr\"{o}m, and Yi~Zhao, \emph{Triangle-degrees in
  graphs and tetrahedron coverings in 3-graphs}, Combin. Probab. Comput.
  \textbf{30} (2021), no.~2, 175--199. \MR{4225783}

\bibitem[HHLZ25]{HHLZ25}
Jianfeng Hou, Caiyun Hu, Xizhi Liu, and Yixiao Zhang, \emph{Density
  {H}ajnal--{S}zemer{\'e}di theorem for cliques of size four}, arXiv preprint
  arXiv:2501.00801 (2025).

\bibitem[KLPS20]{MR4196784}
Jaehoon Kim, Hong Liu, Oleg Pikhurko, and Maryam Sharifzadeh, \emph{Asymptotic
  structure for the clique density theorem}, Discrete Anal. (2020), Paper No.
  19, 26. \MR{4196784}

\bibitem[LPS20]{MR4089395}
Hong Liu, Oleg Pikhurko, and Katherine Staden, \emph{The exact minimum number
  of triangles in graphs with given order and size}, Forum Math. Pi \textbf{8}
  (2020), e8, 144. \MR{4089395}

\bibitem[Moo68]{M68}
John~W Moon, \emph{On independent complete subgraphs in a graph}, Canadian J.
  Math. \textbf{20} (1968), 95--102.

\bibitem[MT21]{MT21}
Klas Markstr\"{o}m and Carsten Thomassen, \emph{Partite {T}ur\'{a}n-densities
  for complete {$r$}-uniform hypergraphs on {$r+1$} vertices}, J. Comb.
  \textbf{12} (2021), no.~2, 235--245. \MR{4290615}

\bibitem[Nag11]{MR2776822}
Zolt\'an~L\'or\'ant Nagy, \emph{A multipartite version of the {T}ur\'an
  problem---density conditions and eigenvalues}, Electron. J. Combin.
  \textbf{18} (2011), no.~1, Paper 46, 15. \MR{2776822}

\bibitem[NT17]{MR3647823}
Lothar Narins and Tuan Tran, \emph{A density {T}ur\'an theorem}, J. Graph
  Theory \textbf{85} (2017), no.~2, 496--524. \MR{3647823}

\bibitem[Pfe12]{MR2965288}
Florian Pfender, \emph{Complete subgraphs in multipartite graphs},
  Combinatorica \textbf{32} (2012), no.~4, 483--495. \MR{2965288}

\bibitem[Raz08]{R2}
Alexander~A. Razborov, \emph{On the minimal density of triangles in graphs},
  Combin. Probab. Comput. \textbf{17} (2008), no.~4, 603--618. \MR{2433944
  (2009i:05118)}

\bibitem[Rei16]{MR3549620}
Christian Reiher, \emph{The clique density theorem}, Ann. of Math. (2)
  \textbf{184} (2016), no.~3, 683--707. \MR{3549620}

\bibitem[Sim68]{S68}
Mikl{\'o}s Simonovits, \emph{A method for solving extremal problems in graph
  theory, stability problems}, Theory of Graphs (Proc. Colloq., Tihany, 1966),
  1968, pp.~279--319.

\end{thebibliography}
\providecommand{\bysame}{\leavevmode\hbox to3em{\hrulefill}\thinspace}
\providecommand{\MR}{\relax\ifhmode\unskip\space\fi MR }
\providecommand{\MRhref}[2]{%
  \href{http://www.ams.org/mathscinet-getitem?mr=#1}{#2}
}
\providecommand{\href}[2]{#2}

\end{document}